\documentclass[reqno,11pt]{amsart}

\usepackage[utf8]{inputenc}
\usepackage[T1]{fontenc}
\usepackage{amssymb}
\usepackage{amsthm}
\usepackage{amsmath}
\usepackage{latexsym}
\usepackage{times}
\usepackage{xcolor}
\usepackage[mathscr]{eucal}
\usepackage{tikz}
\usetikzlibrary{calc}

\newcommand{\acr}{\newline\indent}

\numberwithin{equation}{section}

\allowdisplaybreaks

\newtheorem{theorem}{Theorem}[section]

\newtheorem{corollary}[theorem]{Corollary}
\newtheorem{remark}[theorem]{Remark}
\newtheorem{definition}[theorem]{Definition}
\newtheorem{example}[theorem]{Example}

\title[The Eguchi Construction Arises from Bi-Lagrangian Structures]{The Eguchi Construction Arises from Bi-Lagrangian Structures}

\author[Emmanuel Gnandi, St\'ephane Puechmorel]{Emmanuel Gnandi*, St\'ephane Puechmorel**}

\address{\llap{*\,}INSA de Toulouse\acr
	D\'epartement de G\'enie Math\'ematique\acr
	Universit\'e de Toulouse\acr
	135 avenue de Rangueil\acr
	31077 Toulouse cedex 4\acr
	France}
\email{kpanteemmanuel@gmail.com, gnandi@insa-toulouse.fr}

\address{\llap{**\,}ENACLab\acr
	Ecole Nationale de l'Aviation Civile\acr
	7, Avenue Edouard Belin\acr
	31055 Toulouse\acr
	France}
\email{stephane.puechmorel@enac.fr}



\begin{document}

\begin{abstract}
We propose a new geometric interpretation of Eguchi's construction, which associates to a statistical divergence, an essential ingredient of Information Geometry, a statistical structure on a differentiable manifold. Our approach is based on bi-Lagrangian geometry. Building on the work of O.~E.~Barndorff-Nielsen and P.~E.~Jupp~\cite{Barndorff}, we show that a statistical divergence naturally induces a bi-Lagrangian structure, hence a para-K\"ahler structure. Within this framework, the Riemannian metric $g^D$ and the pair of dual connections $(\nabla^D,{}^*\nabla^D)$ of Eguchi's construction are induced by the para-K\"ahler structure: $g^D$ is the restriction to the diagonal of the para-K\"ahler metric $G$ of signature $(d,d)$, while $\nabla^D$ and ${}^*\nabla^D$ are obtained, through the two natural projections, from the Bott connections of the two transverse Lagrangian foliations, which are themselves the restrictions of the Levi-Civita connection of $G$ (the Hess, K\"unneth or bi-Lagrangian connection).

This perspective shows that the statistical structure associated with a divergence is not merely the formal outcome of successive differentiations, but the manifestation of an underlying bi-Lagrangian geometry. Our approach thus provides an intrinsic and unified understanding of Eguchi's construction, highlights the connections between information geometry, symplectic geometry and para-K\"ahler geometry, and opens new research perspectives at the interface of these fields.
\end{abstract}

\maketitle

\begin{center}
\large\itshape
Dedicated to the memory of H\^ong V\^an L\^e, in recognition of her contributions to Information Geometry. We remember her with gratitude and affection.

\vspace{0.2cm}

Dedicated to the memory of Paulette Libermann, who in 1952 introduced almost paracomplex and paracomplex structures, in analogy with the almost complex and complex cases, as well as para-Hermitian and para-K\"ahler metrics, the paracomplex counterparts of Hermitian and K\"ahler geometry. These structures play a central role throughout this paper.
\end{center}

\section{Introduction}\label{sec:intro}

Information geometry studies families of probability distributions through the tools of differential geometry. Under suitable regularity conditions, a parametric statistical model carries a natural Riemannian metric, the Fisher--Rao metric~\cite{Rao1945}, together with a one-parameter family of affine connections~\cite{Amari2012,Chentsov1982}. Abstracting these data, Lauritzen~\cite{Lauritzen1987} introduced statistical manifolds: Riemannian manifolds endowed with a pair of torsion-free connections that are dual with respect to the metric. By a result of L\^e~\cite{Le2006}, every compact statistical manifold can be realized as a finite-dimensional statistical model.

A second source of statistical structures is provided by statistical divergences, also called contrast functions~\cite{Eguchi} or yokes~\cite{Barndorff}. These are functions of two points of a manifold that measure their dissimilarity, the Kullback--Leibler and Bregman divergences being prominent examples. Divergences serve as loss functions in variational problems, with applications ranging from imaging to machine learning, and play a central role in information geometry, statistical inference, optimization and optimal transport.

\emph{Eguchi's construction}~\cite{Eguchi,Eguchi1992} associates to a divergence $D\colon M\times M\to\mathbb{R}$ a statistical structure $(g^D,\nabla^D,{}^*\nabla^D)$ on $M$ by taking derivatives of $D$ with respect to its two arguments and evaluating them along the diagonal of $M\times M$. Here $\partial_i$ and $\partial'_j$ denote derivatives with respect to the first and second arguments, respectively. The metric is the negative mixed second derivative,
\[
  g^D_{ij}=-\left.\partial_i\partial'_jD\right|_{y=x}.
\]
The Christoffel symbols of the connection $\nabla^D$ are given by a third derivative involving two derivatives in the first variable and one in the second,
\[
  \Gamma^D_{ij,k}=-\left.\partial_i\partial_j\partial'_kD\right|_{y=x},
\]
while those of its dual ${}^*\nabla^D$ are given by the third derivative with the roles of the variables exchanged,
\[
  \Gamma^{*D}_{ij,k}=-\left.\partial_k\partial'_i\partial'_jD\right|_{y=x}.
\]
The two connections are torsion-free and dual with respect to $g^D$, and their average is the Levi-Civita connection of $g^D$. This construction has played a major role in the geometric formalization of Information Geometry.

Conversely, Matumoto~\cite{Matumoto1993} showed that every statistical manifold admits a divergence inducing its structure; this divergence is not unique. When $D$ is a Bregman divergence, the induced structure is dually flat, i.e.\ $\nabla^D$ and ${}^*\nabla^D$ are both flat~\cite{Amari,nagaoka2026differential}. On a dually flat manifold, Amari and Nagaoka introduced a canonical divergence, which is a Bregman divergence~\cite{Amari,nagaoka2026differential}.

Eguchi's construction is usually presented in local coordinates, the metric and the dual connections being obtained by successive differentiation of the divergence. A natural question is whether it admits an intrinsic geometric interpretation. In this paper we show that it arises from an underlying bi-Lagrangian geometry.

Our starting point is the work of Barndorff-Nielsen and Jupp~\cite{Barndorff}, who relate yokes to symplectic geometry. The mixed Hessian of a divergence $D$ defines a symplectic form $\omega_D$ on a neighborhood $\mathcal U$ of the diagonal $\Delta_M$ in $M\times M$, and the two projections of $M\times M$ onto its factors define two transverse Lagrangian foliations of $\mathcal U$ (see Figure~\ref{fig:bi-lagrangian-neighborhood}). Thus a divergence induces a bi-Lagrangian structure. Since bi-Lagrangian structures are equivalent to para-K\"ahler structures~\cite{etayo2006geometry,Gordejuela}, $\mathcal U$ carries a pseudo-Riemannian metric $G$ of neutral signature, whose Levi-Civita connection is the Hess connection~\cite{Hess,Hess2005}. The statistical structure of Eguchi's construction is recovered from this para-K\"ahler geometry by restriction to the diagonal and through the two projections.

Our main results are the following. Theorem~\ref{thm:eguchi} gives a proof of Eguchi's theorem based on the bi-Lagrangian structure induced by a divergence: the metric is the restriction of $G$ to the diagonal, and the dual connections are induced by the Hess connection. Through the two projections, we express these connections in terms of the Bott connections of the two foliations, and show that the resulting structure coincides with Eguchi's (Corollary~\ref{cor:eguchi_coincidence}). Conversely, Theorem~\ref{thm:equivalence_statistical_bilagrangian} shows that every statistical structure on $M$ induces a bi-Lagrangian structure on a neighborhood of the diagonal, and gives conditions under which, conversely, a bi-Lagrangian structure on a neighborhood of the diagonal induces a statistical structure on $M$ in the sense of Kurose.

The paper is organized as follows. Section~\ref{sec:basics} recalls the basic notions: statistical manifolds with examples, bi-Lagrangian structures and the Hess connection. Section~\ref{sec:Result} contains the main results.

\section{Basic Notions}\label{sec:basics}

All manifolds in this paper are assumed to be connected and without boundary, unless otherwise stated.

\subsection{Statistical manifolds}

We recall some basic facts about statistical manifolds; see~\cite{Amari2012,Lauritzen1987} for details. The notion of statistical structure provides the geometric framework of information geometry.

Consider the set of parametric densities
\[
  \mathcal{P}_d(\Omega)
  =\left\{
    p(\cdot\,;\theta)\;\middle|\;
    \begin{aligned}
      &p(x;\theta)>0,\\
      &\int_{\Omega}p(x;\theta)\,dx=1,\\
      &\theta\in\Theta
    \end{aligned}
  \right\},
\]
where $\Theta\subset\mathbb{R}^d$ is open, $\Omega$ is a sample space and $\theta=(\theta^1,\dots,\theta^d)$ are the parameters. Following Amari~\cite{Amari2012}, $\mathcal{P}_d(\Omega)$ is called a $d$-dimensional \emph{statistical model}. We assume that $\theta\mapsto p(\,\cdot\,;\theta)$ is one-to-one, so that $\theta$ is a global coordinate system. Under suitable regularity conditions, Rao~\cite{Rao1945} showed that $\mathcal{P}_d(\Omega)$ is a Riemannian manifold for the Fisher information metric $g_F$, with components
\[
  g^F_{ij}(\theta)
  =4\int_{\Omega}
  \frac{\partial\sqrt{p(x;\theta)}}{\partial\theta^i}\,
  \frac{\partial\sqrt{p(x;\theta)}}{\partial\theta^j}\,dx.
\]

Amari~\cite{Amari2012} and \v{C}encov~\cite{Chentsov1982} introduced a one-parameter family of $\alpha$-connections $\nabla^{(\alpha)}$, $\alpha\in\mathbb{R}$, satisfying
\[
  \nabla^{(\alpha)}g_F=\alpha\,C_{\mathrm{A\text{-}C}},
\]
where $C_{\mathrm{A\text{-}C}}$ is the totally symmetric Amari--\v{C}encov cubic tensor,
\[
  C_{ijk}(\theta)
  =\int_{\Omega}
  \partial_i\log p(x;\theta)\,
  \partial_j\log p(x;\theta)\,
  \partial_k\log p(x;\theta)\,
  p(x;\theta)\,dx.
\]
In local coordinates,
\[
  \Gamma^{(\alpha)}_{ijk}=\Gamma^{(0)}_{ijk}-\frac{\alpha}{2}\,C_{ijk},
\]
where $\Gamma^{(0)}_{ijk}$ are the Christoffel symbols of the Levi-Civita connection of $g_F$. Since $C_{ijk}$ is totally symmetric, each $\nabla^{(\alpha)}$ is torsion-free, $\nabla^{(0)}$ is the Levi-Civita connection of $g_F$, and $\nabla^{(\alpha)}$ and $\nabla^{(-\alpha)}$ are dual with respect to $g_F$.

Inspired by this construction, Lauritzen~\cite{Lauritzen1987} introduced the following notion.

\begin{definition}[\cite{Lauritzen1987}]\label{def:stat_manifold}
  A \emph{statistical manifold} is a triplet $(M,g,\nabla)$, where $g$ is a Riemannian metric and $\nabla$ a torsion-free affine connection such that $\nabla g$ is totally symmetric. Equivalently, it is a quadruplet $(M,g,\nabla,\nabla^*)$, where $\nabla$ and $\nabla^*$ are torsion-free connections dual with respect to $g$:
  \begin{equation}\label{eq:dual_connection}
    X\,g(Y,Z)=g(\nabla_XY,Z)+g(Y,\nabla^*_XZ),
    \qquad \forall\,X,Y,Z\in\Gamma(TM).
  \end{equation}
\end{definition}

Statistical manifolds are also called Codazzi manifolds; see~\cite{Delanoe1989,nomizu1992notes}. The pair $(\nabla,\nabla^*)$ of \emph{statistical connections} leads to the \emph{skewness tensor} $C$,
\begin{equation}\label{eq:skewness_def}
  C(X,Y,Z)=g(\nabla_XY-\nabla^*_XY,\,Z)=-(\nabla_Xg)(Y,Z),
\end{equation}
which is totally symmetric. The curvature tensors satisfy the duality relation
\begin{equation}\label{eq:curvature_duality}
  g\bigl(R^\nabla(Z,W)X,\,Y\bigr)+g\bigl(R^{\nabla^*}(Z,W)Y,\,X\bigr)=0,
  \qquad \forall\,X,Y,Z,W\in\Gamma(TM).
\end{equation}

Statistical models are statistical manifolds. A fundamental question raised by Lauritzen and Amari is whether every statistical manifold is a finite-dimensional statistical model. Hong Van L\^e~\cite{Le2006} answered affirmatively: combining Nash's embedding theorem~\cite{Nash1956,Nash1954}, applied to the metric, with Gromov's immersion theorem~\cite{Gromov1986}, applied to the Amari--\v{C}encov tensor, she showed that every compact statistical manifold admits such a realization.

Statistical structures are closely related to symplectic geometry. In~\cite{Delanoe1989}, P.~Delano\"e showed that, for a torsion-free connection $\nabla$ on a Riemannian manifold $(M,g)$, the triple $(M,g,\nabla)$ is a statistical manifold if and only if the image of the horizontal subbundle $H\subset T(TM)$ defined by $\nabla$ under the tangent map
\[
  T\mathcal{X}_g\colon T(TM)\longrightarrow T(T^{*}M)
\]
is a Lagrangian subbundle of $T(T^{*}M)$, where $T^*M$ carries its canonical symplectic form and $\mathcal{X}_g\colon v\mapsto g(v,\cdot)$. Dombrowski~\cite{Dombrowski1962} and Satoh~\cite{Satoh2019} further constructed an almost K\"ahler structure on the tangent bundle of a statistical manifold.

Among statistical structures, the dually flat ones~\cite{Amari}, also called Hessian structures~\cite{Shima2007}, play an important role in information geometry and optimization. By~\eqref{eq:curvature_duality}, if $\nabla$ is flat then so is $\nabla^{*}$; $(M,g,\nabla,\nabla^{*})$ is then a \emph{dually flat} or \emph{Hessian manifold}. Every compact dually flat manifold has an infinite fundamental group~\cite{AyTuschmann2002}; in particular $\mathbb{S}^{l}$, $l>1$, admits no dually flat structure. A classification of compact orientable $3$-dimensional Hessian manifolds is given in~\cite{Gnandi2026}.

\subsection{Examples of statistical manifolds}

\begin{example}[Trivial statistical manifold]\label{ex:trivial}
  A statistical manifold is \emph{trivial} if $\nabla=\nabla^{*}=\nabla^{\mathrm{LC}}$, the Levi-Civita connection of $g$.
\end{example}

\begin{example}[Univariate normal distributions]\label{ex:normal}
  Let $M=\{(\mu,\sigma)\in\mathbb{R}^2:\sigma>0\}$ with the metric $g:=(d\mu^{2}+2\,d\sigma^{2})/\sigma^{2}$. Let $\nabla$ be the flat connection for which $\theta^1:=\mu/\sigma^{2}$, $\theta^2:=-1/(2\sigma^{2})$ is an affine chart, and $\nabla^{*}$ its dual with respect to $g$, for which $\eta_1:=\mu$, $\eta_2:=\mu^{2}+\sigma^{2}$ is an affine chart. Then $(M,g,\nabla,\nabla^{*})$ is a dually flat manifold.
\end{example}

\begin{example}[Almost contact structure]\label{ex:almost_contact}
  Let $(M,g)$ be a Riemannian manifold and $\xi$ a nowhere vanishing vector field. Then
  \[
    \nabla_XY:=\nabla^{\mathrm{LC}}_XY+g(X,\xi)\,g(Y,\xi)\,\xi,\qquad
    \nabla^{*}_XY:=\nabla^{\mathrm{LC}}_XY-g(X,\xi)\,g(Y,\xi)\,\xi
  \]
  define a statistical manifold $(M,g,\nabla,\nabla^{*})$. This applies in particular to almost contact metric manifolds, which we recall. Let $\dim M=2n+1$.
  \begin{enumerate}
    \item\label{item:acs} An \emph{almost contact structure} is a triple $(\varphi,\xi,\eta)$, with $\varphi$ a $(1,1)$-tensor field, $\xi$ a vector field and $\eta$ a $1$-form, such that
    \begin{align}
      \eta(\xi)&=1, \label{eq:eta-xi}\\
      \eta\circ\varphi&=0, \label{eq:eta-phi}\\
      \varphi^{2}&=-\mathrm{id}+\eta\otimes\xi. \label{eq:phi2}
    \end{align}
    \item Given a pseudo-Riemannian metric $g$, the quadruple $(\varphi,\xi,\eta,g)$ is an \emph{almost contact metric structure} if $(\varphi,\xi,\eta)$ is as in item~\ref{item:acs} and, for some $\varepsilon\in\{\pm1\}$,
    \begin{align}
      g(\xi,\xi)&=\varepsilon, \qquad \eta=\varepsilon\,\xi^{\flat}, \label{eq:g-xi-xi}\\
      g(\varphi X,\varphi Y)&=g(X,Y)-\varepsilon\,\eta(X)\eta(Y),
        \qquad X,Y\in\mathfrak{X}(M). \label{eq:g-phi-compat}
    \end{align}
    Equivalently, it is a reduction of the frame bundle of $M$ to $U(n)\times\{1\}$.
  \end{enumerate}
  For such a structure and $\beta\in\mathbb{R}$, set
  \[
    \nabla^{\pm\beta}_XY:=\nabla^{\mathrm{LC}}_XY\mp\beta\,\eta(X)\eta(Y)\,\xi.
  \]
  Then $(M,g,\nabla^{\beta},\nabla^{-\beta})$ is a statistical manifold: $\nabla^{\pm\beta}$ are torsion-free, mutually dual with respect to $g$, and $(\nabla^{\beta}_Xg)(Y,Z)=2\beta\varepsilon\,\eta(X)\eta(Y)\eta(Z)$ is totally symmetric.
\end{example}

\begin{example}[Symplectic manifold]\label{ex:sy}
  A \emph{symplectic connection}~\cite{Bieliavsky2006} on a symplectic manifold $(M,\omega)$ is a torsion-free connection for which $\omega$ is parallel. Let $\nabla^0$ be a torsion-free connection and $N$ the tensor defined by
  \[
    \nabla^0_X\omega(Y,Z)=:\omega\bigl(N(X,Y),Z\bigr),\qquad X,Y,Z\in\mathfrak{X}(M).
  \]
  Since $\omega$ is closed, $\circlearrowleft_{XYZ}\,\omega\bigl(N(X,Y),Z\bigr)=0$, where $\circlearrowleft$ denotes the sum over cyclic permutations. Then
  \[
    \nabla_XY:=\nabla^0_XY+\tfrac13N(X,Y)+\tfrac13N(Y,X)
  \]
  is a symplectic connection.

  Any other connection reads $\nabla'_XY=\nabla_XY+A(X)Y$, with $A$ a $1$-form valued in $\mathrm{End}(TM)$. It is torsion-free iff $A(X)Y=A(Y)X$, and symplectic if furthermore
  \[
    0=\nabla'_X\omega(Y,Z)=-\omega\bigl(A(X)Y,Z\bigr)-\omega\bigl(Y,A(X)Z\bigr),
  \]
  that is, iff $\widehat A(X,Y,Z):=\omega\bigl(A(X)Y,Z\bigr)$ is totally symmetric. Thus symplectic connections form an affine space modeled on totally symmetric $3$-tensors.

  Fix a compatible almost complex structure $J$, set $g(X,Y)=\omega(X,JY)$, and let $\nabla^{\mathrm{LC}}$ be its Levi-Civita connection. For $a\in\mathbb{R}$ put
  \[
    \nabla^{\pm a}_XY=\nabla^{\mathrm{LC}}_XY\mp\tfrac a2\,J\bigl(A(X)Y\bigr).
  \]
  Then
  \begin{align*}
    \nabla^a+\nabla^{-a}&=2\nabla^{\mathrm{LC}},\\
    X\,g(Y,Z)&=g(\nabla^a_XY,Z)+g(Y,\nabla^{-a}_XZ),
  \end{align*}
  and $(g,\nabla^a,\nabla^{-a})$ is a statistical structure on $M$, non-trivial whenever $a\neq0$ and $A\neq0$.
\end{example}

We conclude with the pseudo-Riemannian generalization due to Kurose~\cite{Kurose1994}. Let $G$ be a pseudo-Riemannian metric on $M$ and let $\nabla$ be a torsion-free connection. The pair $(G,\nabla)$ is called a \emph{statistical structure} on $M$ when the covariant derivative $\nabla G$ is totally symmetric. This condition generalizes Lauritzen's definition to the pseudo-Riemannian setting.

\subsection{Bi-Lagrangian structures and the Hess connection}

Bi-Lagrangian structures lie at the crossroads of symplectic geometry, affine differential geometry, web geometry, pseudo-Riemannian geometry and foliation theory. They are related to Teichm\"uller theory~\cite{loustau2017bi}, to geometric quantization in the sense of Kostant--Souriau~\cite{simms1976lectures}, and play a role in mathematical physics. Moreover, symplectic connections are a key tool for the formal deformation quantization of Flato, Lichnerowicz and Sternheimer~\cite{bayen1977quantum,flato1976crochet}.

A strong link is provided by para-K\"ahler structures, introduced by P.~Libermann~\cite{libermann1954probleme,libermann1952structures}. F.~Etayo Gordejuela and his coauthors~\cite{etayo2006geometry,Gordejuela} established the equivalence between para-K\"ahler and bi-Lagrangian structures. Vaisman introduced generalized para-K\"ahler structures~\cite{vaisman2015generalized}, and Leichtnam, Tang and Weinstein para-K\"ahler Lie algebroids~\cite{leichtnam2007poisson}.

\begin{definition}
  A \emph{Lagrangian submanifold} of a symplectic manifold $(M^{2n},\omega)$ is an $n$-dimensional submanifold on which $\omega$ vanishes. A foliation whose leaves are Lagrangian submanifolds is a \emph{Lagrangian foliation}.
\end{definition}

\begin{definition}
  A \emph{bi-Lagrangian structure} on a symplectic manifold $(M,\omega)$ is a pair $(\mathcal{F}_1,\mathcal{F}_2)$ of transverse Lagrangian foliations. It is also called a Lagrangian $2$-web~\cite{tabachnikov1993geometry}, a K\"unneth structure~\cite{guillemin2000kunneth,hamilton2023kunneth}, or a bipolarized symplectic structure~\cite{kanai1988geodesic}.
\end{definition}

\begin{definition}
  A \emph{para-K\"ahler structure} on $M$ is a triple $(G,F,\omega)$, where $G$ is a pseudo-Riemannian metric, $F$ a para-complex structure ($F^2=\mathrm{id}$), and $\omega=G(F\cdot,\cdot)$ a symplectic form, such that $\nabla^{G}F=0$, where $\nabla^{G}$ is the Levi-Civita connection of $G$.
\end{definition}

A bi-Lagrangian structure carries a distinguished connection, the \emph{Hess connection}, also called the \emph{K\"unneth} or \emph{bi-Lagrangian connection}~\cite{Hess,Hess2005}. It is the unique symplectic connection ($\nabla\omega=0$) preserving both foliations $\mathcal{F}_1$ and $\mathcal{F}_2$ and with vanishing mixed torsion,
\[
  T(X,Z)=0,\qquad X\in\Gamma(\mathcal{F}_1),\ Z\in\Gamma(\mathcal{F}_2).
\]

Let $(M,\omega,\mathcal{F}_1,\mathcal{F}_2)$ be a bi-Lagrangian manifold, so that $TM=\mathcal{F}_1\oplus\mathcal{F}_2$. Following Hess~\cite{Hess,Hess2005}, for $X,Y\in\Gamma(\mathcal{F}_1)$ and $Z,W\in\Gamma(\mathcal{F}_2)$,
\begin{equation}\label{eq:hess}
  \begin{aligned}
    \nabla^{\omega}_{X}Y &= \operatorname{pr}_1\,\mathcal{X}_{\omega}^{-1}\,\mathcal{L}_{X}\,\mathcal{X}_{\omega}(Y),
    &\qquad
    \nabla^{\omega}_{Z}Y &= \operatorname{pr}_1\,[Z,Y],\\[10pt]
    \nabla^{\omega}_{Z}W &= \operatorname{pr}_2\,\mathcal{X}_{\omega}^{-1}\,\mathcal{L}_{Z}\,\mathcal{X}_{\omega}(W),
    &\qquad
    \nabla^{\omega}_{X}W &= \operatorname{pr}_2\,[X,W],
  \end{aligned}
\end{equation}
where $\mathcal{X}_{\omega}\colon\mathfrak{X}(M)\to\Omega^{1}(M)$, $\mathcal{X}_{\omega}(X)=\iota_X\omega$, is the isomorphism induced by $\omega$, and $\operatorname{pr}_i\colon TM\to\mathcal{F}_i$ is the projection associated with $TM=\mathcal{F}_1\oplus\mathcal{F}_2$.

\begin{remark}
  Hess connections give a way to construct symplectic connections on an arbitrary symplectic manifold $(M,\omega)$. By Darboux's theorem, every point has a neighborhood $U$ with coordinates $(q^i,p_i)$ in which
  \[
  \omega|_U=\mathrm{d}p_i\wedge\mathrm{d}q^i .
  \]
  The foliations $\mathcal{F}_1=\{p_i=\mathrm{const}\}$ and $\mathcal{F}_2=\{q^i=\mathrm{const}\}$ form a pair of transverse Lagrangian foliations; therefore $(\omega|_U,\mathcal{F}_1,\mathcal{F}_2)$ is a local bi-Lagrangian structure on $U$, with Hess connection $\nabla^{\mathrm{Hess}}_U$~\cite{Hess}. For a locally finite cover $(U_j)_j$ by such neighborhoods and a subordinate partition of unity $(\phi_j)_j$, the convex combination
  \[
  \nabla=\sum_j\phi_j\,\nabla^{\mathrm{Hess}}_{U_j}
  \]
  is a global symplectic connection.
\end{remark}

\section{Eguchi's Construction, Bi-Lagrangian Structures, and Para-K\"ahler Structures}
\label{sec:Result}

\subsection{Divergences and Eguchi's construction}

Statistical divergences measure the dissimilarity between probability distributions, and many of them have been introduced in information geometry. They are in general asymmetric and do not satisfy the triangle inequality; for a comprehensive account, see Amari~\cite{Amari}, Nielsen~\cite{nielsen2020elementary} and Calin--Udri\c{s}te~\cite{Calin2014}. In a broad sense, every distance is a symmetric divergence satisfying the triangle inequality; a notable class of divergences that are not distances is obtained by squaring a distance. Prominent examples of divergence functions include the Kullback--Leibler divergence~\cite{Amari}, the Bregman divergence~\cite{Bregman}, the $f$-divergences~\cite{Csiszar,AliSilvey,CsiszarTop}, the $\Phi$-divergence~\cite{Zhang}, the Kagan divergence~\cite{Kagan}, the Chernoff information~\cite{Chernoff}, the Jeffreys divergence~\cite{RobertChopinRousseau}, and exponential contrast functions. Divergences play a fundamental role in information theory~\cite{Arndt2012}, statistics~\cite{Eguchi2022}, machine learning~\cite{Wang2022}, quantitative finance~\cite{Taylor2019}, signal processing~\cite{Basseville2013} and neuroscience.

\medskip
Geometrically, a divergence historically called a contrast function~\cite{Eguchi} or a yoke~\cite{Barndorff} is a smooth function $D\colon M\times M\to\mathbb{R},$
written $D(x:y)$, such that:
\begin{enumerate}
  \item $D(x:y)\geq 0$ for all $x,y\in M$, with equality if and only if $x=y$;
  \item $\partial_iD(x:y)\big|_{y=x}=0$ and $\partial'_jD(x:y)\big|_{y=x}=0$ for all $i,j$, where $\partial_i=\partial/\partial x^i$ and $\partial'_j=\partial/\partial y^j$;
  \item the tensor
  \[
  g_{ij}(x)=-\partial_i\partial'_jD(x:y)\big|_{y=x}
  \]
  is positive definite and therefore defines a Riemannian metric on $M$.
\end{enumerate}

\subsubsection*{Examples}

\emph{Squared distances.} For $x,y\in\mathbb{R}^n$, the squared Euclidean and Mahalanobis distances are
\begin{align*}
  D_E(x\,\|\,y)&=\|x-y\|_2^2=x^\top x+y^\top y-2x^\top y,\\
  D_M(x\,\|\,y)&=\|x-y\|_{S^{-1}}^2=x^\top S^{-1}x+y^\top S^{-1}y-2x^\top S^{-1}y,
\end{align*}
where $S$ is symmetric positive definite. For $A,B\in\mathbb{R}^{m\times n}$, the squared Frobenius distance is
\[
  D_F(A\,\|\,B)=\|A-B\|_F^2
  =\operatorname{tr}\bigl(A^\top A+B^\top B-2A^\top B\bigr).
\]

\smallskip
\emph{Burg divergence.} For symmetric positive-definite matrices $S,T\in\mathbb{R}^{n\times n}$,
\[
  D_{\mathrm{Burg}}(S\,\|\,T)
  =\operatorname{tr}\bigl(ST^{-1}\bigr)
  -\log\det\bigl(ST^{-1}\bigr)-n.
\]

\smallskip
\emph{Kullback--Leibler divergence of Gaussian distributions.}
For $\mu,\mu'\in\mathbb{R}^n$ and symmetric positive-definite covariance matrices $\Sigma,\Sigma'$,
\begin{align*}
  D_{\mathrm{NKL}}\bigl((\mu,\Sigma)\,\|\,(\mu',\Sigma')\bigr)
  &=D_{\mathrm{KL}}\bigl(
      f_{\mathcal N(\mu,\Sigma)}\,\big\|\,
      f_{\mathcal N(\mu',\Sigma')}
    \bigr)\\
  &=\frac12\Bigl[
      (\mu-\mu')^\top\Sigma'^{-1}(\mu-\mu')
      +\operatorname{tr}\bigl(\Sigma\Sigma'^{-1}\bigr)
      -\log\det\bigl(\Sigma\Sigma'^{-1}\bigr)
      -n
    \Bigr],
\end{align*}
where
\[
  D_{\mathrm{KL}}\bigl(f\,\|\,f'\bigr)
  \equiv\mathbb{E}_f\!\left[\ln\frac{f(X)}{f'(X)}\right]
  =\int f(x)\,\ln\frac{f(x)}{f'(x)}\,dx .
\]

\smallskip
\emph{Bregman divergence.} The Bregman divergence induced by a differentiable strictly convex function $\psi$ is defined as the gap between the value of $\psi$ at $m$ and the value at $m$ of the plane tangent to $\psi$ at $n$:
\[
  D^{\mathrm{B}}_{\psi}(m\,\|\,n)
  \equiv
  \psi(m)
  -
  \bigl[
    \psi(n)
    +
    \nabla_n\psi(n)^\top(m-n)
  \bigr].
\]
\subsubsection*{Eguchi's construction}

Eguchi~\cite{Eguchi,Eguchi1992} showed that the second- and third-order derivatives of a divergence $D$ canonically define a statistical structure $(g^D,\nabla^D,{}^*\nabla^D)$ on $M$:
\begin{align}
  g^D(X,Y)&=-D[X\mid Y], \label{eq:metric}\\
  g^D\bigl(\nabla^D_XY,Z\bigr)&=-D[XY\mid Z], \label{eq:connection1}\\
  g^D\bigl({}^*\nabla^D_XY,Z\bigr)&=-D[Z\mid XY]. \label{eq:connection2}
\end{align}
Here, for $X_1,\dots,X_n,Y_1,\dots,Y_s\in\mathfrak{X}(M)$, the vector fields $X_i$ differentiate $D$ in its first variable, while the $Y_j$ differentiate it in its second variable:
\[
  D[X_1\cdots X_n\mid Y_1\cdots Y_s](p)
  =
  (X_1)_p\cdots(X_n)_p\,
  (Y_1)_q\cdots(Y_s)_q\,
  D(p,q)\big|_{q=p}.
\]
The connections $\nabla^D$ and ${}^*\nabla^D$ are torsion-free and dual with respect to $g^D$; that is,
\[
  X\bigl(g^D(Y,Z)\bigr)
  =
  g^D\bigl(\nabla^D_XY,Z\bigr)
  +g^D\bigl(Y,{}^*\nabla^D_XZ\bigr),
  \qquad X,Y,Z\in\mathfrak{X}(M).
\]
Moreover, the Levi-Civita connection of $g^D$ is
\[
  \nabla^{\mathrm{LC}}
  =\frac12\bigl(\nabla^D+{}^*\nabla^D\bigr).
\]

\medskip
In local coordinates, the metric and the Christoffel symbols of the first kind,
\[
  \Gamma^D_{ij,k}
  =g^D\bigl(\nabla^D_{\partial_i}\partial_j,\partial_k\bigr),
  \qquad
  \Gamma^{*D}_{ij,k}
  =g^D\bigl({}^*\nabla^D_{\partial_i}\partial_j,\partial_k\bigr),
\]
are given by~\cite{Eguchi,Eguchi1992}
\begin{align*}
  g^D_{ij}(x)
  &=-\left.\partial_i\partial'_jD(x:y)\right|_{y=x},\\
  \Gamma^D_{ij,k}(x)
  &=-\left.\partial_i\partial_j\partial'_kD(x:y)\right|_{y=x},\\
  \Gamma^{*D}_{ij,k}(x)
  &=-\left.\partial_k\partial'_i\partial'_jD(x:y)\right|_{y=x}.
\end{align*}
Eguchi thus gives a constructive proof that every divergence induces a statistical structure on the underlying manifold: he explicitly defines the Riemannian metric and the pair of affine connections, and proves that these connections are dual with respect to the metric.

\subsubsection*{Our viewpoint}

Our approach to Eguchi's theorem is based on a different geometric viewpoint, namely bi-Lagrangian, or para-K\"ahler, geometry. It makes apparent the symplectic structure underlying information geometry. The statistical structure induced by a divergence arises from the natural bi-Lagrangian structure on a neighborhood of the diagonal in $M\times M$. Rather than constructing the metric and the two dual connections separately, we derive them from this ambient geometry. Consequently, the statistical structure appears as a consequence of the bi-Lagrangian structure, rather than as a formal consequence of successive differentiations of $D$. This reveals a deep relationship between Eguchi's construction and para-K\"ahler geometry.

\subsection{A bi-Lagrangian proof of Eguchi's theorem}

\begin{theorem}[An Alternative Proof of Eguchi's Theorem]\label{thm:eguchi}
  Let $M$ be a $d$-dimensional differentiable manifold and
  $D\colon M\times M\to\mathbb{R}$ a divergence. Then $D$ induces a
  statistical structure $(g,\nabla,\nabla^*)$ on $M$.
\end{theorem}

\begin{proof}

The idea is to work on a neighborhood of the diagonal in $M\times M$, where $D$
induces a bi-Lagrangian structure, and then to restrict to the diagonal in
order to recover the statistical structure on $M$.

\medskip
\emph{Step 1.}
Let $x=(x^1,\dots,x^d)$ and $y=(y^1,\dots,y^d)$ be local coordinates around
$p$ and $q$, respectively, so that $(x,y)$ are product coordinates on
$M\times M$ around $(p,q)$. Following Barndorff-Nielsen and
Jupp~\cite{Barndorff} (Section~3), consider the map
\[
  \varphi_D\colon M\times M\to T^*M,
  \qquad
  (x,y)\longmapsto
  \bigl(x,\,d_xD(\cdot,y)\bigr).
\]
Let $\theta$ be the Liouville $1$-form on $T^*M$ and let
$\omega_{\mathrm{can}}=-d\theta$ be the canonical symplectic form. Then
\[
  \omega_D:=\varphi_D^*\omega_{\mathrm{can}}
  =-d(\varphi_D^*\theta)
  =-\sum_{i,j=1}^d
  \frac{\partial^2D}{\partial x^i\partial y^j}(x,y)\,
  dx^i\wedge dy^j
\]
is exact. By the defining properties of a divergence, the mixed Hessian matrix
\[
  \left(
  -\frac{\partial^2D}{\partial x^i\partial y^j}(x,y)
  \right)_{i,j}
\]
is nondegenerate along the diagonal
\[
  \Delta_M=\{(p,p)\mid p\in M\}\subset M\times M,
\]
which is the graph of the identity map $\mathrm{id}_M\colon M\to M$. Hence
$\omega_D$ is symplectic on an open neighborhood
$\mathcal U\subset M\times M$ of $\Delta_M$.

The projections
\[
  p_1(x,y)=x,
  \qquad
  p_2(x,y)=y
\]
define two transverse foliations of $\mathcal U$:
\[
  \mathcal F_1
  =
  \bigl\{(\{x\}\times M)\cap\mathcal U\bigr\}_{x\in M},
  \qquad
  \mathcal F_2
  =
  \bigl\{(M\times\{y\})\cap\mathcal U\bigr\}_{y\in M}.
\]
Their tangent distributions are
\[
  T\mathcal F_1
  =
  \operatorname{span}\left\{
  \frac{\partial}{\partial y^1},\dots,
  \frac{\partial}{\partial y^d}
  \right\},
  \qquad
  T\mathcal F_2
  =
  \operatorname{span}\left\{
  \frac{\partial}{\partial x^1},\dots,
  \frac{\partial}{\partial x^d}
  \right\}
\]
(see Figure~\ref{fig:bi-lagrangian-neighborhood}). Since $\omega_D$ involves
only the mixed terms $dx^i\wedge dy^j$, it vanishes on pairs of vectors
tangent to the same foliation; equivalently,
\[
  \omega_D\big|_{T\mathcal F_1\times T\mathcal F_1}=0,
  \qquad
  \omega_D\big|_{T\mathcal F_2\times T\mathcal F_2}=0.
\]
Therefore $(\omega_D,\mathcal F_1,\mathcal F_2)$ is a bi-Lagrangian structure
on $\mathcal U$, and it is this structure that carries the whole construction.

\medskip

\emph{Step 2.}
By~\cite{Bryant2001,Gordejuela,Alekseevsky2025}, the bi-Lagrangian structure
$(\omega_D,\mathcal F_1,\mathcal F_2)$ determines a para-complex structure
$K\in\Gamma(\operatorname{End}(T\mathcal U))$ satisfying
\[
  K^2=\operatorname{Id},
  \qquad
  K\big|_{T\mathcal F_1}=-\operatorname{Id},
  \qquad
  K\big|_{T\mathcal F_2}=+\operatorname{Id}.
\]
It also determines the pseudo-Riemannian metric
\[
  G(X,Y):=\omega_D(KX,Y),
\]
which is symmetric and has neutral signature $(d,d)$. Moreover,
\[
  G\big|_{T\mathcal F_1\times T\mathcal F_1}=0,
  \qquad
  G\big|_{T\mathcal F_2\times T\mathcal F_2}=0,
\]
so $\mathcal F_1$ and $\mathcal F_2$ are $G$-null foliations.

By Hess's theorem~\cite{Hess}, there exists a unique symplectic connection
$\nabla^{\omega_D}$ on $(\mathcal U,\omega_D)$ which preserves both
$T\mathcal F_1$ and $T\mathcal F_2$; it coincides with the Levi-Civita
connection of $G$~\cite{Hess2005,Kaneyuki}.

Since $T\mathcal U=T\mathcal F_1\oplus T\mathcal F_2$, formula~\eqref{eq:hess}
applies to $(\mathcal U,\omega_D,\mathcal F_1,\mathcal F_2)$, with
\[
  \omega=\omega_D,
  \qquad
  \mathcal X_\omega=\mathcal X_{\omega_D},
  \qquad
  \operatorname{pr}_i=\operatorname{pr}_{T\mathcal F_i}.
\]
Thus, for $X,Y\in\Gamma(T\mathcal F_1)$ and
$Z,W\in\Gamma(T\mathcal F_2)$,
\begin{equation}\label{eq:hess_D}
  \begin{aligned}
    \nabla^{\omega_D}_{X}Y
      &=\operatorname{pr}_1\,
        \mathcal X_{\omega_D}^{-1}\,
        \mathcal L_X\,\mathcal X_{\omega_D}(Y),
    &\qquad
    \nabla^{\omega_D}_{Z}Y
      &=\operatorname{pr}_1\bigl[Z,Y\bigr],\\[6pt]
    \nabla^{\omega_D}_{Z}W
      &=\operatorname{pr}_2\,
        \mathcal X_{\omega_D}^{-1}\,
        \mathcal L_Z\,\mathcal X_{\omega_D}(W),
    &\qquad
    \nabla^{\omega_D}_{X}W
      &=\operatorname{pr}_2\bigl[X,W\bigr].
  \end{aligned}
\end{equation}
Since $\nabla^{\omega_D}$ preserves $T\mathcal F_1$ and $T\mathcal F_2$, its
restrictions to vector fields tangent to the leaves,
\[
  \nabla^{\mathcal F_1}_XY:=\nabla^{\omega_D}_XY,
  \qquad
  \nabla^{\mathcal F_2}_ZW:=\nabla^{\omega_D}_ZW,
\]
define connections on $T\mathcal F_1$ and $T\mathcal F_2$, respectively.

\medskip

\emph{Step 3.}
The bi-Lagrangian data on $\mathcal U$ suffice to produce a metric on $M$
itself. For $X\in\Gamma(TM)$, define its horizontal and vertical lifts to
$M\times M$ by
\[
  X^h=(X\circ p_1,0)\in\Gamma(T\mathcal F_2),
  \qquad
  X^v=(0,X\circ p_2)\in\Gamma(T\mathcal F_1).
\]
In local coordinates, for $X=X^i(x)\,\partial_{x^i}$, this reads
\[
  X^h=X^i(x)\,\partial_{x^i},
  \qquad
  X^v=X^i(y)\,\partial_{y^i}.
\]
Define
\[
  g(X,Y):=
  \left.G\bigl(X^{v},Y^{h}\bigr)\right|_{\Delta_M}.
\]

For $X=\partial_{x^i}$ and $Y=\partial_{x^j}$, since
$K|_{T\mathcal F_1}=-\operatorname{Id}$, we have
\[
  G\bigl(X^{v},Y^{h}\bigr)
  =
  \omega_D\bigl(KX^{v},Y^{h}\bigr)
  =
  \omega_D\bigl(-\partial_{y^i},\partial_{x^j}\bigr)
  =
  -\frac{\partial^2D}{\partial x^j\,\partial y^i}(x,y).
\]
Evaluating along the diagonal and using the equality of mixed second partial
derivatives of the smooth function $D$, we obtain
\[
  \left.
  \frac{\partial^2D}{\partial x^j\,\partial y^i}(x,y)
  \right|_{y=x}
  =
  \left.
  \frac{\partial^2D}{\partial x^i\,\partial y^j}(x,y)
  \right|_{y=x}.
\]
Hence
\[
  g_{ij}(x)
  =
  -\left.
  \frac{\partial^2D}{\partial x^i\,\partial y^j}(x,y)
  \right|_{y=x}.
\]
By the defining properties of a divergence, the matrix
$\bigl(g_{ij}(x)\bigr)$ is symmetric and positive definite. Therefore $g$ is a
Riemannian metric on $M$.

\medskip

\emph{Step 4.}
We now translate the affine geometry of $M$ into that of $\mathcal U$. For
$X,Y,Z\in\Gamma(TM)$, their horizontal and vertical lifts satisfy
\begin{equation}\label{eq:lift_brackets}
  [X^h,Y^h]=[X,Y]^h,
  \qquad
  [X^v,Y^v]=[X,Y]^v,
  \qquad
  [X^h,Y^v]=0.
\end{equation}

By~\eqref{eq:hess_D} and~\eqref{eq:lift_brackets}, the mixed terms of the Hess
connection vanish on lifted vector fields:
\[
  \nabla^{\omega_D}_{X^h}Y^v
  =\operatorname{pr}_1[X^h,Y^v]=0,
  \qquad
  \nabla^{\omega_D}_{X^v}Z^h
  =\operatorname{pr}_2[X^v,Z^h]=0.
\]
Since $\nabla^{\omega_D}$ is the Levi-Civita connection of $G$, it satisfies
$\nabla^{\omega_D}G=0$. Combining this with the previous vanishing
statements, we obtain
\begin{equation}\label{eq:parallel_G}
  (X^v+X^h)\cdot G(Y^v,Z^h)
  =
  G\bigl(\nabla^{\mathcal F_1}_{X^v}Y^v,Z^h\bigr)
  +G\bigl(Y^v,\nabla^{\mathcal F_2}_{X^h}Z^h\bigr).
\end{equation}
This suggests defining two connections $\nabla$ and $\nabla^*$ on $M$ by
\begin{align}
  g(\nabla_XY,Z)
    &:=G\bigl(\nabla^{\mathcal F_1}_{X^v}Y^v,Z^h\bigr)
       \big|_{\Delta_M},
    \label{eq:nabla_bott}\\
  g(\nabla^*_XY,Z)
    &:=G\bigl(Z^v,\nabla^{\mathcal F_2}_{X^h}Y^h\bigr)
       \big|_{\Delta_M}.
    \label{eq:nablastar_bott}
\end{align}

\smallskip
It remains to check that~\eqref{eq:nabla_bott}
and~\eqref{eq:nablastar_bott} define connections. Let $f\in C^\infty(M)$. A
direct computation gives
\[
  (fX)^h=(f\circ p_1)\,X^h,
  \qquad
  (fX)^v=(f\circ p_2)\,X^v,
\]
and
\[
  X^h(f\circ p_1)=(Xf)\circ p_1,
  \qquad
  X^v(f\circ p_2)=(Xf)\circ p_2.
\]
Moreover, both $f\circ p_1$ and $f\circ p_2$ restrict to $f$ on $\Delta_M$.

We verify the connection axioms for $\nabla$; the argument for $\nabla^*$ is
analogous, using the $C^\infty(\mathcal U)$-linearity and the Leibniz rule of
$\nabla^{\mathcal F_2}$, together with $(f\circ p_1)|_{\Delta_M}=f$ and
$X^h(f\circ p_1)=(Xf)\circ p_1$. Additivity in $X$ and $Y$ follows
immediately from the additivity of $\nabla^{\mathcal F_1}$ and of the lifts.

\begin{enumerate}
  \item[(a)] \emph{$C^\infty(M)$-linearity in $X$.}
  Since $(fX)^v=(f\circ p_2)X^v$ and $\nabla^{\mathcal F_1}$ is
  $C^\infty(\mathcal U)$-linear in its first argument,
  \begin{align*}
    g(\nabla_{fX}Y,Z)
    &=G\bigl(\nabla^{\mathcal F_1}_{(f\circ p_2)X^v}Y^v,Z^h\bigr)
      \big|_{\Delta_M}\\
    &=(f\circ p_2)\,
      G\bigl(\nabla^{\mathcal F_1}_{X^v}Y^v,Z^h\bigr)
      \big|_{\Delta_M}\\
    &=f\,g(\nabla_XY,Z).
  \end{align*}
  Hence $\nabla_{fX}Y=f\,\nabla_XY$.

  \item[(b)] \emph{Leibniz rule in $Y$.}
  Since $(fY)^v=(f\circ p_2)Y^v$ and $\nabla^{\mathcal F_1}$ satisfies the
  Leibniz rule in its second argument,
  \begin{align*}
    g(\nabla_X(fY),Z)
    &=G\bigl(
        \nabla^{\mathcal F_1}_{X^v}\bigl((f\circ p_2)Y^v\bigr),
        Z^h
      \bigr)\Big|_{\Delta_M}\\
    &=G\bigl(
        X^v(f\circ p_2)\,Y^v
        +(f\circ p_2)\,\nabla^{\mathcal F_1}_{X^v}Y^v,
        Z^h
      \bigr)\Big|_{\Delta_M}\\
    &=(Xf)\,g(Y,Z)+f\,g(\nabla_XY,Z),
  \end{align*}
  where we used $X^v(f\circ p_2)=(Xf)\circ p_2$ and
  $(f\circ p_2)|_{\Delta_M}=f$. Hence
  \[
    \nabla_X(fY)=(Xf)\,Y+f\,\nabla_XY.
  \]
\end{enumerate}

\medskip

\emph{Step 5.}
We now relate the two connections through $g$. Along the diagonal, in product
notation,
\[
  (X^h+X^v)_{(x,x)}=(X_x,X_x)=d\iota_x(X_x),
  \qquad
  \iota(x)=(x,x).
\]
Thus $X^h+X^v$ is tangent to $\Delta_M$ and corresponds to $X$ under the
identification $\Delta_M\simeq M$. Since
\[
  G(Y^v,Z^h)\big|_{\Delta_M}=g(Y,Z),
\]
identity~\eqref{eq:parallel_G}, together with the definitions
\eqref{eq:nabla_bott} and~\eqref{eq:nablastar_bott}, yields
\[
  X\bigl(g(Y,Z)\bigr)
  =
  g(\nabla_XY,Z)+g(Y,\nabla^*_XZ),
  \qquad
  X,Y,Z\in\Gamma(TM).
\]
Hence $\nabla$ and $\nabla^*$ are dual with respect to $g$.

\medskip

\emph{Step 6.}
It remains to show that $\nabla$ and $\nabla^*$ are torsion-free, so that
$(g,\nabla,\nabla^*)$ is indeed a statistical structure. Since the Hess
connection $\nabla^{\omega_D}$ is torsion-free and preserves both
$T\mathcal F_1$ and $T\mathcal F_2$, its restrictions to the foliations are
torsion-free as well:
\begin{align}
  \nabla^{\mathcal F_1}_{X^v}Y^v
  -\nabla^{\mathcal F_1}_{Y^v}X^v
  &=[X^v,Y^v],
  \label{eq:hess_torsion}\\
  \nabla^{\mathcal F_2}_{Z^h}W^h
  -\nabla^{\mathcal F_2}_{W^h}Z^h
  &=[Z^h,W^h].
  \label{eq:hess_torsion2}
\end{align}

We transfer these identities to $M$ by means of the lifts. Using
\eqref{eq:hess_torsion} and~\eqref{eq:lift_brackets}, we obtain
\begin{align*}
  g(\nabla_XY-\nabla_YX,Z)
  &=G\bigl(
      \nabla^{\mathcal F_1}_{X^v}Y^v
      -\nabla^{\mathcal F_1}_{Y^v}X^v,
      Z^h
    \bigr)\Big|_{\Delta_M}\\
  &=G\bigl([X^v,Y^v],Z^h\bigr)\big|_{\Delta_M}\\
  &=G\bigl([X,Y]^v,Z^h\bigr)\big|_{\Delta_M}\\
  &=g\bigl([X,Y],Z\bigr).
\end{align*}
Since $g$ is nondegenerate, it follows that
\[
  \nabla_XY-\nabla_YX=[X,Y],
\]
so $\nabla$ is torsion-free.

Similarly, using \eqref{eq:hess_torsion2} and~\eqref{eq:lift_brackets}, we get
\begin{align*}
  g(Y,\nabla^*_XZ-\nabla^*_ZX)
  &=G\bigl(
      Y^v,
      \nabla^{\mathcal F_2}_{X^h}Z^h
      -\nabla^{\mathcal F_2}_{Z^h}X^h
    \bigr)\Big|_{\Delta_M}\\
  &=G\bigl(Y^v,[X^h,Z^h]\bigr)\big|_{\Delta_M}\\
  &=G\bigl(Y^v,[X,Z]^h\bigr)\big|_{\Delta_M}\\
  &=g\bigl(Y,[X,Z]\bigr).
\end{align*}
Hence
\[
  \nabla^*_XZ-\nabla^*_ZX=[X,Z],
\]
and $\nabla^*$ is torsion-free as well.

Combining this with the duality established in Step~5, we conclude that
$(g,\nabla,\nabla^*)$ is a statistical structure on $M$. This completes the
proof.
\end{proof}

\begin{remark}\label{rem:kappa}
Via the splitting $T\mathcal U=T\mathcal F_1\oplus T\mathcal F_2$, we identify
\[
  (T\mathcal U/T\mathcal F_1)^*\simeq T^*\mathcal F_2,
  \qquad
  (T\mathcal U/T\mathcal F_2)^*\simeq T^*\mathcal F_1 .
\]
Restricting $\mathcal X_{\omega_D}\colon T\mathcal U\to T^*\mathcal U$ to each Lagrangian leaf yields canonical isomorphisms
\[
  \kappa_1\colon T\mathcal F_1\xrightarrow{\ \sim\ }T^*\mathcal F_2,
  \qquad
  \kappa_2\colon T\mathcal F_2\xrightarrow{\ \sim\ }T^*\mathcal F_1 .
\]
\end{remark}

\begin{remark}\label{rem:iso}
Let $\{i,j\}=\{1,2\}$ and $\operatorname{pr}_i=\operatorname{pr}_{T\mathcal F_i}$, so that $T\mathcal F_j$ is complementary to $T\mathcal F_i$. For $X,Y\in\Gamma(T\mathcal F_i)$, the Bott connection of $\mathcal F_i$ is given by
\[
  \nabla^{\mathcal F_i}_XY
  =\operatorname{pr}_i\,\mathcal X_{\omega_D}^{-1}\bigl(\iota_X\,d(\iota_Y\omega_D)\bigr)
  =\kappa_i^{-1}\bigl(\mathcal L_X(\kappa_i Y)\bigr).
\]
Moreover, for $Z\in\Gamma(T\mathcal F_j)$,
\begin{equation}\label{eq:bott}
  \omega_D\bigl(\nabla^{\mathcal F_i}_XY,Z\bigr)
  =\mathcal L_X\bigl(\omega_D(Y,Z)\bigr)
  +\omega_D\bigl(\operatorname{pr}_j[X,Z],Y\bigr),
\end{equation}
which is precisely the defining identity of the Bott connection of $\mathcal F_i$.
\end{remark}

\subsection{Projection formulas and comparison with Eguchi's construction}

The identities of Remark~\ref{rem:iso} express the whole statistical structure in terms of $\omega_D$.

\begin{corollary}\label{cor:encoded}
The statistical structure $(g,\nabla,\nabla^*)$ on $M$ is entirely encoded by $\omega_D$: for all $X,Y,Z\in\Gamma(TM)$,
\[
  g(Y,Z)=\omega_D\bigl(Y^h,Z^v\bigr)\big|_{\Delta_M},
\]
and
\begin{align}
  g(\nabla_XY,Z)   &=\mathcal L_{X^h}\bigl(\omega_D(Y^h,Z^v)\bigr)\big|_{\Delta_M},
    \label{eq:3.10}\\
  g(Y,\nabla^*_XZ) &=\mathcal L_{X^v}\bigl(\omega_D(Y^h,Z^v)\bigr)\big|_{\Delta_M}.
    \label{eq:3.11}
\end{align}
Here the Lie derivatives are taken on $\mathcal U$ and the result is restricted to $\Delta_M$.
\end{corollary}

\begin{proof}
\emph{Step 1: the metric.} Since $KV=-V$ for $V\in T\mathcal F_1$ and $KW=W$ for $W\in T\mathcal F_2$,
\[
  G(V,W)=-\omega_D(V,W)=\omega_D(W,V).
\]
As $g(Y,Z)=G(Y^v,Z^h)|_{\Delta_M}$, we get $g(Y,Z)=\omega_D(Z^h,Y^v)|_{\Delta_M}$, and the symmetry of $g$ gives $g(Y,Z)=\omega_D(Y^h,Z^v)|_{\Delta_M}$.

\emph{Step 2: formula~\eqref{eq:3.10}.} By definition of $\nabla$ and Step~1,
\[
  g(\nabla_XY,Z)
  =G\bigl(Z^v,\nabla^{\mathcal F_2}_{X^h}Y^h\bigr)\big|_{\Delta_M}
  =\omega_D\bigl(\nabla^{\mathcal F_2}_{X^h}Y^h,Z^v\bigr)\big|_{\Delta_M}.
\]
Applying~\eqref{eq:bott} with $i=2$, $j=1$ and $[X^h,Z^v]=0$ (by~\eqref{eq:lift_brackets}),
\[
  \omega_D\bigl(\nabla^{\mathcal F_2}_{X^h}Y^h,Z^v\bigr)
  =\mathcal L_{X^h}\bigl(\omega_D(Y^h,Z^v)\bigr),
\]
which is~\eqref{eq:3.10}.

\emph{Step 3: formula~\eqref{eq:3.11}.} By definition of $\nabla^*$ and the symmetry of $g$,
\[
  g(Y,\nabla^*_XZ)
  =G\bigl(\nabla^{\mathcal F_1}_{X^v}Z^v,Y^h\bigr)\big|_{\Delta_M}
  =-\,\omega_D\bigl(\nabla^{\mathcal F_1}_{X^v}Z^v,Y^h\bigr)\big|_{\Delta_M}.
\]
Applying~\eqref{eq:bott} with $i=1$, $j=2$ and $[X^v,Y^h]=0$,
\[
  \omega_D\bigl(\nabla^{\mathcal F_1}_{X^v}Z^v,Y^h\bigr)
  =\mathcal L_{X^v}\bigl(\omega_D(Z^v,Y^h)\bigr).
\]
Hence, by antisymmetry of $\omega_D$,
\[
  g(Y,\nabla^*_XZ)
  =-\,\mathcal L_{X^v}\bigl(\omega_D(Z^v,Y^h)\bigr)\big|_{\Delta_M}
  =\mathcal L_{X^v}\bigl(\omega_D(Y^h,Z^v)\bigr)\big|_{\Delta_M},
\]
which is~\eqref{eq:3.11}. This completes the
proof.
\end{proof}

The statistical connections of Theorem~\ref{thm:eguchi} are determined by the Bott connections $\nabla^{\mathcal F_1}$, $\nabla^{\mathcal F_2}$ through the projections $p_1$, $p_2$.

\begin{corollary}[Projection formulas for the Bott connections]\label{cor:projection}
For all $X,Y\in\Gamma(TM)$,
\begin{align}
  \nabla^*_XY
    &=p_{2*}\bigl(\nabla^{\mathcal F_1}_{X^v}Y^v\bigr),
    \label{eq:3.6}\\
  \nabla_XY
    &=p_{1*}\bigl(\nabla^{\mathcal F_2}_{X^h}Y^h\bigr),
    \label{eq:3.7}\\
  \nabla^{\mathrm{LC}}_XY
    &=\tfrac12\,p_{2*}\bigl(\nabla^{\mathcal F_1}_{X^v}Y^v\bigr)
     +\tfrac12\,p_{1*}\bigl(\nabla^{\mathcal F_2}_{X^h}Y^h\bigr).
    \label{eq:3.8}
\end{align}
\end{corollary}

\begin{proof}
We use only the defining identity
\begin{equation}\label{eq:bridge}
  g(U,Z)=G\bigl(U^v,Z^h\bigr)\big|_{\Delta_M},
  \qquad U,Z\in\Gamma(TM),
\end{equation}
together with the nondegeneracy of $g$.

\emph{Step 1: formula~\eqref{eq:3.6}.} Let $U:=\nabla^*_XY$ and $W:=p_{2*}\bigl(\nabla^{\mathcal F_1}_{X^v}Y^v\bigr)$. By~\eqref{eq:nablastar_bott},
\[
  g(U,Z)=G\bigl(\nabla^{\mathcal F_1}_{X^v}Y^v,Z^h\bigr)\big|_{\Delta_M}.
\]
Since $\nabla^{\mathcal F_1}_{X^v}Y^v$ is the vertical lift of $W$,~\eqref{eq:bridge} gives $g(U,Z)=g(W,Z)$ for all $Z$, hence $U=W$.

\emph{Step 2: formula~\eqref{eq:3.7}.} Let $V:=p_{1*}\bigl(\nabla^{\mathcal F_2}_{X^h}Y^h\bigr)$. By~\eqref{eq:nabla_bott}, the antisymmetry of $\omega_D$ and $G(V,W)=-\omega_D(V,W)$ for $V\in T\mathcal F_1$, $W\in T\mathcal F_2$,
\[
  g(Z,\nabla_XY)
  =G\bigl(Z^v,\nabla^{\mathcal F_2}_{X^h}Y^h\bigr)\big|_{\Delta_M}
  =G\bigl(Z^v,V^h\bigr)\big|_{\Delta_M}
  =g(Z,V)\qquad\forall Z.
\]
By symmetry and nondegeneracy of $g$, $\nabla_XY=V$.

\emph{Step 3: formula~\eqref{eq:3.8}.} By Theorem~\ref{thm:eguchi}, $\nabla$ and $\nabla^*$ are torsion-free and $g$-dual, so their average is torsion-free and metric, hence equals $\nabla^{\mathrm{LC}}$. Conclude with~\eqref{eq:3.6} and~\eqref{eq:3.7}. This completes the
proof.
\end{proof}

By Remark~3.2, the symplectic form \(\omega_D\) induces the bundle isomorphisms
\[
\begin{aligned}
\kappa_1:T\mathcal F_1&\longrightarrow T^*\mathcal F_2,
&\kappa_1(V)(W)&=\omega_D(V,W),\\
\kappa_2:T\mathcal F_2&\longrightarrow T^*\mathcal F_1,
&\kappa_2(W)(V)&=\omega_D(W,V),
\end{aligned}
\]
for \(V\in T\mathcal F_1\) and \(W\in T\mathcal F_2\). These maps are bundle isomorphisms because \(\omega_D\) is symplectic and \(T\mathcal F_1\) and \(T\mathcal F_2\) are complementary Lagrangian distributions.

By Remark~3.3, the Bott connections are expressed in terms of these isomorphisms as
\[
\begin{aligned}
\nabla^{\mathcal F_1}_{X^v}Y^v
&=\kappa_1^{-1}\bigl(\mathcal L_{X^v}(\kappa_1Y^v)\bigr),\\
\nabla^{\mathcal F_2}_{X^h}Y^h
&=\kappa_2^{-1}\bigl(\mathcal L_{X^h}(\kappa_2Y^h)\bigr).
\end{aligned}
\]
Substituting these expressions into the projection formulas of Corollary~3.5 yields
\[
{}\nabla^*_XY
=p_{2*}\kappa_1^{-1}
\bigl(\mathcal L_{X^v}(\kappa_1Y^v)\bigr),
\]
and
\[
\nabla_XY
=p_{1*}\kappa_2^{-1}
\bigl(\mathcal L_{X^h}(\kappa_2Y^h)\bigr).
\]
Finally, since the Levi-Civita connection satisfies
\(\nabla^{LC}=\tfrac12(\nabla+\nabla^*)\), we obtain
\[
\begin{aligned}
\nabla^{LC}_XY
={}&\frac12p_{2*}\kappa_1^{-1}
\bigl(\mathcal L_{X^v}(\kappa_1Y^v)\bigr)\\
&+\frac12p_{1*}\kappa_2^{-1}
\bigl(\mathcal L_{X^h}(\kappa_2Y^h)\bigr).
\end{aligned}
\]

We now show that our construction recovers Eguchi’s construction.

\begin{corollary}\label{cor:eguchi_coincidence}
The metric $g$ and the connections $\nabla,\nabla^*$ of Theorem~\ref{thm:eguchi} coincide with Eguchi's:
\[
  (g,\nabla,\nabla^*)=(g^{D},\nabla^{D},{}^*\nabla^{D}).
\]
\end{corollary}

\begin{proof}
Let $(x^i,y^{\bar\jmath})$ be local product coordinates on $\mathcal U$ and
\[
  \omega_{i\bar\jmath}:=-\frac{\partial^2D}{\partial x^i\partial y^{\bar\jmath}}.
\]
Then $\omega_D=\omega_{i\bar\jmath}\,dx^i\wedge dy^{\bar\jmath}$, and since $G(V,W)=\omega_D(W,V)$ for $V\in T\mathcal F_1$, $W\in T\mathcal F_2$, we have $G(\partial_i,\partial_{\bar\jmath})=\omega_{i\bar\jmath}$.

\emph{Step 1: the metric.} Restricting to $\Delta_M$,
\[
  g_{ij}=-\left.\frac{\partial^2D}{\partial x^i\partial y^{\bar\jmath}}\right|_{y=x},
\]
so $g=g^D$.

 \emph{Step 2: the connection $\nabla$.} Let $X=\partial_i$, $Y=\partial_j$, and $Z=\partial_k$. Since $\omega_D(Y^h,Z^v)=\omega_{j\bar k}$,~\eqref{eq:3.10} gives
\[
  g(\nabla_{\partial_i}\partial_j,\partial_k)
  =\left.\frac{\partial}{\partial x^i}
  \omega_D(\partial_j^h,\partial_k^v)\right|_{y=x}
  =\left.\frac{\partial\omega_{j\bar k}}{\partial x^i}\right|_{y=x}
  =-\left.\frac{\partial^3D}{\partial x^i\partial x^j\partial y^k}\right|_{y=x}.
\]
Consequently,
\[
  \Gamma_{ij,k}
  =-\left.\partial_i\partial_j\partial'_kD\right|_{y=x}
  =\Gamma^D_{ij,k},
\]
where $\partial'_k=\frac{\partial}{\partial y^k}$. Hence $\nabla=\nabla^D$.

\emph{Step 3: the dual connection $\nabla^*$.} With the same coordinate vector fields,~\eqref{eq:3.11} gives
\[
  g(\nabla^*_{\partial_i}\partial_k,\partial_j)
  =\left.\frac{\partial}{\partial y^i}
  \omega_D(\partial_j^h,\partial_k^v)\right|_{y=x}
  =\left.\frac{\partial\omega_{j\bar k}}{\partial y^i}\right|_{y=x}
  =-\left.\frac{\partial^3D}{\partial y^i\partial x^j\partial y^k}\right|_{y=x}.
\]
Consequently,
\[
  \Gamma^*_{ik,j}
  =-\left.\partial_j\partial'_i\partial'_kD\right|_{y=x}
  =\Gamma^{*D}_{ik,j},
\]
where $\partial'_i=\frac{\partial}{\partial y^i}$. Hence $\nabla^*={}^*\nabla^D$.
\end{proof}

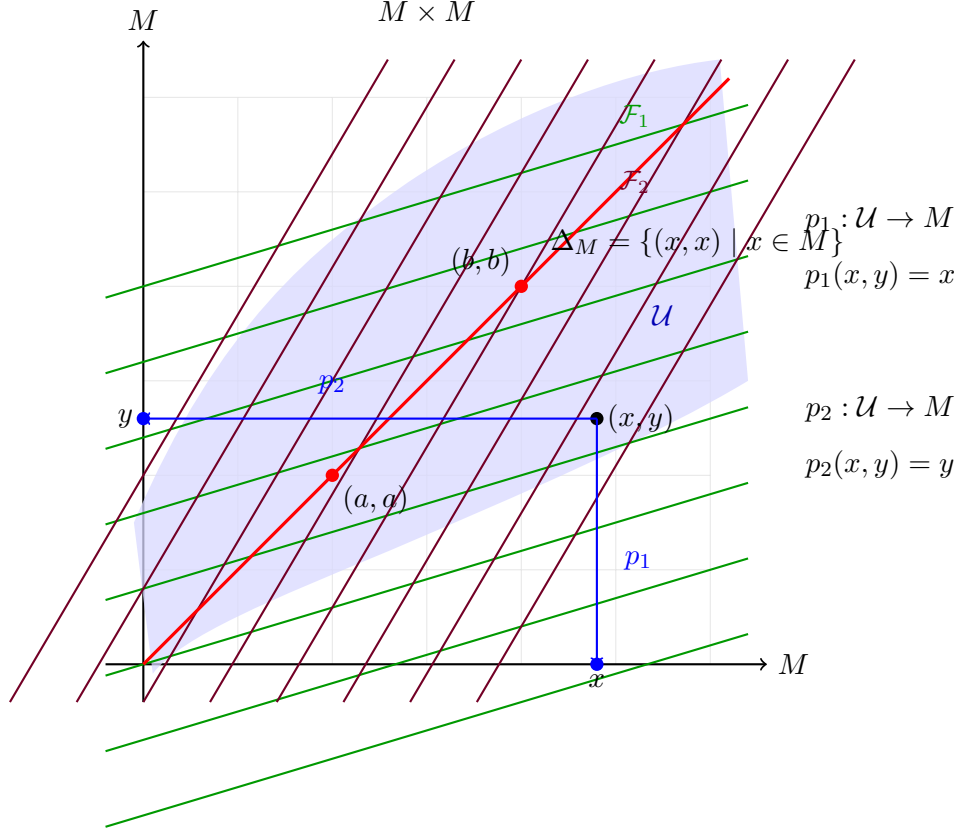
\begin{figure}[htbp]
\centering
\begin{tikzpicture}[scale=2.5]
\draw[gray!20, very thin] (0,0) grid[step=0.5] (3,3);
\draw[->, thick] (-0.2,0) -- (3.3,0) node[right] {$M$};
\draw[->, thick] (0,-0.2) -- (0,3.3) node[above] {$M$};
\node at (1.5,3.45) {$M\times M$};

\fill[blue!15,opacity=.75]
(-0.05,0.75)
.. controls (0.4,1.8) and (0.9,2.2) ..
(1.4,2.55)
.. controls (2.0,2.95) and (2.6,3.15) ..
(3.05,3.2)
--
(3.2,1.5)
.. controls (2.7,1.2) and (2.0,0.9) ..
(1.3,0.6)
.. controls (0.7,0.35) and (0.2,0.15) ..
(0.05,-0.05)
-- cycle;
\node[blue!70!black] at (2.75,1.85) {$\mathcal{U}$};

\foreach \c in {-0.8,-0.4,0,0.4,0.8,1.2,1.6,2.0}
{
  \draw[green!60!black, thick]
    (-0.2,{0.3*(-0.2)+\c}) -- (3.2,{0.3*(3.2)+\c});
}
\foreach \c in {-3.2,-2.6,-2.0,-1.4,-0.8,-0.2,0.4,1.0}
{
  \draw[purple!60!black, thick]
    ({(-0.2-\c)/1.7},-0.2) -- ({(3.2-\c)/1.7},3.2);
}

\draw[red, very thick] (0,0) -- (3.1,3.1);
\node[above right] at (2.1,2.1) {$\Delta_M=\{(x,x)\mid x\in M\}$};
\fill[red] (1,1) circle (1pt);
\node[below right] at (1,1) {$(a,a)$};
\fill[red] (2,2) circle (1pt);
\node[above left] at (2,2) {$(b,b)$};

\fill[black] (2.4,1.3) circle (1pt);
\node[right] at (2.4,1.3) {$(x,y)$};

\draw[->,blue,thick] (2.4,1.3) -- (2.4,0);
\node[blue] at (2.62,0.55) {$p_1$};
\fill[blue] (2.4,0) circle (1pt);
\node[below] at (2.4,0) {$x$};

\draw[->,blue,thick] (2.4,1.3) -- (0,1.3);
\node[blue] at (1.0,1.48) {$p_2$};
\fill[blue] (0,1.3) circle (1pt);
\node[left] at (0,1.3) {$y$};

\node[align=left] at (3.9,2.2)
  {$\displaystyle p_1:\mathcal{U}\to M$\\[2mm]$\displaystyle p_1(x,y)=x$};
\node[align=left] at (3.9,1.2)
  {$\displaystyle p_2:\mathcal{U}\to M$\\[2mm]$\displaystyle p_2(x,y)=y$};
\node[green!60!black] at (2.6,2.9) {\small $\mathcal{F}_1$};
\node[purple!60!black] at (2.6,2.55) {\small $\mathcal{F}_2$};
\end{tikzpicture}

\caption{A neighborhood $\mathcal{U}$ of the diagonal $\Delta_M\subset M\times M$ with the two transverse foliations $\mathcal{F}_1$ and $\mathcal{F}_2$.}
\label{fig:bi-lagrangian-neighborhood}
\end{figure}

\subsection{Statistical structures and bi-Lagrangian structures}

To conclude this section, we present a general result relating statistical structures to bi-Lagrangian (equivalently, para-Kähler) structures.

\begin{theorem}\label{thm:equivalence_statistical_bilagrangian}
Every statistical structure $(g,\nabla,\nabla^*)$ on $M$ is induced, through a divergence $D$, by a bi-Lagrangian structure $(\omega_D,\mathcal F_1,\mathcal F_2)$ on a neighborhood of $\Delta_M\subset M\times M$. Conversely, every bi-Lagrangian structure $(\omega,\mathcal F_1,\mathcal F_2)$ on a neighborhood $\mathcal U$ of $\Delta_M$ for which $\Delta_M$ is Lagrangian and transverse to $\mathcal F_1$ and $\mathcal F_2$ induces a statistical structure $(g,\nabla,\nabla^*)$ on $M$ in the sense of Kurose.
\end{theorem}

\begin{proof}
By~\cite{Matumoto1993}, every statistical manifold $(M,g,\nabla,\nabla^*)$ admits a divergence $D\colon M\times M\to\mathbb{R}$ (not uniquely) such that $g(X,Y)=-D[X\mid Y]$, $g(\nabla_XY,Z)=-D[XY\mid Z]$ and $g(Y,\nabla^*_XZ)=-D[Y\mid XZ]$ for all $X,Y,Z\in\mathfrak{X}(M)$. Following~\cite[Section~3]{Barndorff}, the $2$-form $\omega_D=\sum_{i,j}-\bigl(\partial^2D/\partial x^i\partial y^{\bar\jmath}\bigr)\,dx^i\wedge dy^{\bar\jmath}$ is symplectic on an open neighborhood $\mathcal U$ of $\Delta_M$, by nondegeneracy of the mixed Hessian. The fibers of $p_1$ and $p_2$ are Lagrangian for $\omega_D$, which gives the bi-Lagrangian structure $(\omega_D,\mathcal F_1,\mathcal F_2)$ on $\mathcal U$.

Conversely, let $(\omega,\mathcal F_1,\mathcal F_2)$ be a bi-Lagrangian structure on $\mathcal U\supset\Delta_M$ such that $\Delta_M$ is Lagrangian and transverse to $\mathcal F_1$ and $\mathcal F_2$. Then
\[
  T\mathcal U|_{\Delta_M}=T\Delta_M\oplus T\mathcal F_1=T\Delta_M\oplus T\mathcal F_2 .
\]
Transversality of the foliations and of $\Delta_M$ gives product coordinates $(x,y)$ in which $\Delta_M=\{y=x\}$, so the lifts $X^h,X^v$ and the relations~\eqref{eq:lift_brackets} are those of Step~4 of the proof of Theorem~\ref{thm:eguchi}. Step~2 of that proof uses only the bi-Lagrangian structure; it provides the para-complex structure $K$, the neutral metric $G=\omega(K\cdot,\cdot)$, the Hess connection $\nabla^\omega=\nabla^G$ and the formulas~\eqref{eq:hess_D}. In particular $G(V,W)=\omega(W,V)$ for $V\in T\mathcal F_1$ and $W\in T\mathcal F_2$.

Set $g(X,Y):=G(X^v,Y^h)|_{\Delta_M}=\omega(Y^h,X^v)|_{\Delta_M}$. The field $X^h+X^v$ is tangent to $\Delta_M$, which is Lagrangian, and $\mathcal F_1,\mathcal F_2$ are Lagrangian; hence
\begin{align*}
  0&=\omega\bigl(X^h+X^v,\,Y^h+Y^v\bigr)\big|_{\Delta_M}\\
   &=\omega(X^h,Y^v)\big|_{\Delta_M}+\omega(X^v,Y^h)\big|_{\Delta_M},
\end{align*}
that is, $g(X,Y)=g(Y,X)$. If $g(X,Y)=0$ for every $Y\in\mathfrak X(M)$, then $\omega(X^h,Y^v)=0$ along $\Delta_M$. Since $\Delta_M$ is transverse to $\mathcal F_2$, the vectors $Y^v$ span $T\mathcal F_1|_{\Delta_M}$. Hence $\omega(X^h,V)=0$ for every $V\in T\mathcal F_1$. Since $\omega$ is nondegenerate and $\mathcal F_1,\mathcal F_2$ are transverse Lagrangian foliations, we obtain $X^h=0$. As $X^h$ is the $\mathcal F_2$-component of the tangent vector $X^h+X^v\in T\Delta_M$, and $T\Delta_M\cap T\mathcal F_1=\{0\}$, it follows that $X=0$. Thus $g$ is nondegenerate, and hence a pseudo-Riemannian metric on $M$.

Define $\nabla,\nabla^*$ by~\eqref{eq:nabla_bott} and~\eqref{eq:nablastar_bott}. Step~4 of the proof of Theorem~\ref{thm:eguchi} applies verbatim, since it uses only~\eqref{eq:lift_brackets}, \eqref{eq:hess_D}, the Lagrangian property of $\mathcal F_1$ and the nondegeneracy of $g$ just proved; hence $\nabla$ and $\nabla^*$ are connections. Since $\nabla^\omega G=0$, \eqref{eq:parallel_G} holds, and the argument of Step~5 of that proof gives
\[
  X\,g(Y,Z)=g(\nabla_XY,Z)+g(Y,\nabla^*_XZ).
\]
By~\eqref{eq:hess_torsion}, \eqref{eq:hess_torsion2} and~\eqref{eq:lift_brackets}, as in Step~6 of that proof, $\nabla$ and $\nabla^*$ are torsion-free. By duality and torsion-freeness, and by the equivalence in Definition~\ref{def:stat_manifold}, $\nabla g$ is totally symmetric. Hence $(g,\nabla,\nabla^*)$ is a statistical structure on $M$ in the sense of Kurose. This completes the
proof.
\end{proof}

\section*{Acknowledgments}

The authors are grateful to Professor Jun Zhang (the University of Michigan, Ann Arbor, USA) for fruitful discussions on divergence functions and the Eguchi construction.

The authors are grateful to  Professor A.~Banyaga (Penn State University, USA) for invaluable discussions and support.



\begin{thebibliography}{10}

\bibitem{Barndorff} O. E. Barndorff-Nielsen and P. E. Jupp, Statistics, yokes and symplectic geometry, Annales de la Facult\'e des Sciences de Toulouse: Math\'ematiques, 6 (1997), no.~3, 389--427.

\bibitem{Gordejuela} F. Etayo Gordejuela and R. Santamar\'ia, The canonical connection of a bi-Lagrangian manifold, Journal of Physics A: Mathematical and General, 34 (2001), no.~5, 981--987.

\bibitem{Kaneyuki} S. Kaneyuki, On a remarkable class of homogeneous symplectic manifolds, 1991.

\bibitem{hamilton2023kunneth}
Hamilton, M.J.D., Kotschick, D.: K\"unneth geometry: symplectic manifolds and their Lagrangian foliations. Cambridge University Press, vol. 108 (2023).

\bibitem{Hess} H. Hess, Connections on symplectic manifolds and geometric quantization, in Differential Geometrical Methods in Mathematical Physics (Proc. Conf., Aix-en-Provence and Salamanca, 1979), Springer, 2006, pp.~153--166.

\bibitem{Eguchi} S. Eguchi, Second order efficiency of minimum contrast estimators in a curved exponential family, The Annals of Statistics, 1983, pp.~793--803.

\bibitem{Eguchi1992} S. Eguchi, Geometry of minimum contrast, Hiroshima Mathematical Journal, 22 (1992), no.~3, 631--647.

\bibitem{Dombrowski1962}
Dombrowski, P.: On the geometry of the tangent bundle. J. Reine Angew. Math. (1962).

\bibitem{Satoh2019}
Satoh, H.: Almost Hermitian structures on tangent bundles. arXiv preprint arXiv:1908.10824 (2019).

\bibitem{Hess2005} H. Hess, On a geometric quantization scheme generalizing those of Kostant--Souriau and Czy\.z, in \textit{Differential Geometric Methods in Mathematical Physics}, Springer, 2005, pp.~1--35.

\bibitem{Rao1945}
Rao, C.R.: Information and the accuracy attainable in the estimation of statistical parameters. Bull. Calcutta Math. Soc. 37(3), 81--91 (1945).

\bibitem{tabachnikov1993geometry}
Tabachnikov, S.: Geometry of Lagrangian and Legendrian 2-web. Differ. Geom. Appl. 3(3), 265--284 (1993).

\bibitem{guillemin2000kunneth}
Guillemin, V.: K\"unneth geometry. Seminar talk at Harvard, 4 December 2000.

\bibitem{kanai1988geodesic}
Kanai, M.: Geodesic flows of negatively curved manifolds with smooth stable and unstable foliations. Ergod. Theory Dyn. Syst. 8(2), 215--239 (1988).

\bibitem{Amari} S. Amari, Information geometry and its applications, Springer, 2016.

\bibitem{Bregman} L. M. Bregman, The relaxation method of finding the common point of convex sets and its application to the solution of problems in convex programming, USSR Computational Mathematics and Mathematical Physics, 7 (1967), no.~3, 200--217.

\bibitem{Le2006}
L\^e, H.~V.: Statistical manifolds are statistical models. \textit{J. Geom.} \textbf{84}(1--2), 83--93 (2006).

\bibitem{Alekseevsky2025}
D. Alekseevsky, G. Manno, G. Moreno, Invariant Monge--Amp\`ere equations on contactified para--K\"ahler manifolds. \emph{Ann. Global Anal. Geom.} \textbf{67} (2025), no.~4, Article~25.

\bibitem{Basseville2013} M. Basseville, Divergence measures for statistical data processing---An annotated bibliography, Signal Processing, 93(4), 621--633, 2013.

\bibitem{etayo2006geometry}
Etayo, F., Santamar\'ia, R., Tr\'ias, U.R.: The geometry of a bi-Lagrangian manifold. Differ. Geom. Appl. 24(1), 33--59 (2006).

\bibitem{Eguchi2022} S. Eguchi and O. Komori, Minimum divergence methods in statistical machine learning, Springer, 2022.

\bibitem{Bryant2001} R. Bryant, Bochner-K\"ahler metrics, Journal of the American Mathematical Society, 14(3), 623--715, 2001.

\bibitem{leichtnam2007poisson}
Leichtnam, E., Tang, X., Weinstein, A.: Poisson geometry and deformation quantization near a strictly pseudoconvex boundary. J. Eur. Math. Soc. 9(4), 681--704 (2007).

\bibitem{nomizu1992notes}
Nomizu, K., Simon, U.: Notes on conjugate connections. In: Geometry and Topology of Submanifolds IV, pp. 152--172. World Scientific (1992).

\bibitem{Arndt2012} C. Arndt, Information measures: information and its description in science and engineering, Springer Science \& Business Media, 2012.

\bibitem{Csiszar} I. Csisz\'ar, Information-type measures of difference of probability distributions and indirect observation, Studia Scientiarum Mathematicarum Hungarica, 2 (1967), 229--318.

\bibitem{Nash1956}
Nash, J.: The imbedding problem for Riemannian manifolds. \textit{Ann. Math.} \textbf{63}(1), 20--63 (1956).

\bibitem{AliSilvey} S. M. Ali and S. D. Silvey, A general class of coefficients of divergence of one distribution from another, Journal of the Royal Statistical Society: Series B (Methodological), 28 (1966), no.~1, 131--142.

\bibitem{Zhang} J. Zhang, Divergence function, duality, and convex analysis, Neural Computation, 16 (2004), no.~1, 159--195.

\bibitem{Nash1954}
Nash, J.: $C^1$ isometric imbeddings. \textit{Ann. Math.} \textbf{60}(3), 383--396 (1954).

\bibitem{CsiszarTop} I. Csisz\'ar, On topological properties of $f$-divergence, Studia Sci. Math. Hungar., 2 (1967), 330--339.

\bibitem{vaisman2015generalized}
Vaisman, I.: Generalized para-K\"ahler manifolds. Differ. Geom. Appl. 42, 84--103 (2015).

\bibitem{Chernoff} H. Chernoff, A measure of asymptotic efficiency for tests of a hypothesis based on the sum of observations, The Annals of Mathematical Statistics, 1952, pp.~493--507.

\bibitem{Taylor2019} S. Taylor, Clustering financial return distributions using the Fisher information metric, Entropy, 21(2), 110, 2019.

\bibitem{Kagan} A. M. Kagan, On the theory of Fisher's amount of information, Doklady Akademii Nauk SSSR, 151 (1963), 277--278.

\bibitem{RobertChopinRousseau} C. P. Robert, N. Chopin and J. Rousseau, Harold Jeffreys's theory of probability revisited, 2009.

\bibitem{AyTuschmann2002} N. Ay and W. Tuschmann, Dually flat manifolds and global information geometry, Open Systems \& Information Dynamics, 9 (2002), no.~2, 195--200, Springer.

\bibitem{Lauritzen1987} S. L. Lauritzen, Statistical manifolds, Differential Geometry in Statistical Inference, 10 (1987), no.~2, 163--216, Institute of Mathematical Statistics, Hayward.

\bibitem{Kurose1994} T. Kurose, On the divergences of 1-conformally flat statistical manifolds, Tohoku Mathematical Journal, Second Series, 46 (1994), no.~3, 427--433.

\bibitem{libermann1952structures}
Libermann, R.: Sur les structures presque paracomplexes. C. R. Acad. Sci. Paris 234, 2517--2519 (1952).

\bibitem{libermann1954probleme}
Libermann, P.: Sur le probl\`eme d'\'equivalence de certaines structures infinit\'esimales. Ann. Mat. Pura Appl. 36(1), 27--120 (1954).

\bibitem{flato1976crochet}
Flato, M., Lichnerowicz, A., Sternheimer, D.: Crochet de Moyal-Vey et quantification (1976).

\bibitem{nagaoka2026differential}
Hiroshi Nagaoka and Shun-ichi Amari, ``Differential geometry of smooth families of probability distributions,'' \emph{Information Geometry}, pp. 1--29, 2026, Springer.

\bibitem{bayen1977quantum}
Bayen, F., Flato, M., Fronsdal, C., Lichnerowicz, A., Sternheimer, D.: Quantum mechanics as a deformation of classical mechanics. Lett. Math. Phys. 1(6), 521--530 (1977).

\bibitem{Delanoe1989}
Delano\"e, P.: Remarques sur les vari\'et\'es localement hessiennes (1989).

\bibitem{Bieliavsky2006} P. Bieliavsky, M. Cahen, S. Gutt, J. Rawnsley, and L. Schwachh\"ofer, Symplectic connections, International Journal of Geometric Methods in Modern Physics, 3 (2006), no.~3, 375--420.

\bibitem{Gromov1986}
Gromov, M.: \textit{Partial Differential Relations}, Vol.~9. Springer Science \& Business Media (1986).

\bibitem{Chentsov1982} N. N. Chentsov, Statistical decision rules and optimal inference, American Mathematical Society, 1982.

\bibitem{Wang2022} Q. Wang, Y. Ma, K. Zhao and Y. Tian, A comprehensive survey of loss functions in machine learning, Annals of Data Science, 9(2), 187--212, 2022.

\bibitem{simms1976lectures}
Simms, D.J., Woodhouse, N.M.J.: Lectures on geometric quantization. Springer (1976).

\bibitem{loustau2017bi}
Loustau, B., Sanders, A.: Bi-Lagrangian structures and Teichm\"uller theory. arXiv preprint arXiv:1708.09145 (2017).

\bibitem{Amari2012} S. Amari, Differential-geometrical methods in statistics, Springer Science \& Business Media, 2012.

\bibitem{Matumoto1993} T. Matumoto, Any statistical manifold has a contrast function---On the $C^3$-functions taking the minimum at the diagonal of the product manifold, Hiroshima Mathematical Journal, 23 (1993), no.~2, 327--332.

\bibitem{Shima2007} H. Shima, The geometry of Hessian structures, World Scientific, 2007.

\bibitem{Calin2014} O. Calin and C. Udri{\c{s}}te, Geometric Modeling in Probability and Statistics, Vol.~121, Springer, New York, 2014.

\bibitem{Gnandi2026}
E. Gnandi, \emph{The topology of 3-dimensional Hessian manifolds}, Rev. R. Acad. Cienc. Exactas Fis. Nat. Ser. A Math. \textbf{120} (2026), no.~4, 137. doi:10.1007/s13398-026-01936-z.


\bibitem{nielsen2020elementary}
F. Nielsen,
An elementary introduction to information geometry,
\textit{Entropy}, \textbf{22}(10) (2020), Article 1100.

\end{thebibliography}
\end{document}